\documentclass[11pt]{article}
\usepackage{amsmath,amssymb,amsthm}

\newtheorem{thm}{\noindent Theorem}[section]
\newtheorem{lmm}{\noindent Lemma}[section]
\newtheorem{crl}{\noindent Corollary}[section]

\newtheorem{open}{\noindent Open problem}

\newtheorem{rem}{\noindent Remark}[section]

\begin{document}
\title{Topics and problems on favorite sites of random walks}
\author{Izumi Okada}
\date{\empty}
\maketitle


\begin{abstract}
In this article, we study special points of a simple random walk and a Gaussian free field, such as (nearly) favorite points, late points  and high points. 
In section $2$, we extend results of \cite{okada1} for $d\ge 3$ and suggest open problems for $d=2$. 
In section $3$, we  give a survey on the geometric structures of (nearly) favorite points and late points of a simple random walk and high points of a Gaussian free field  in two dimension. 
\end{abstract}

\section{Introduction}
In this paper, we study properties of the local time of a random walk and special sites among a random walk range or a Gaussian free field. 
Major parts of this paper are survey and contains some minor extensions of existing results. 
As one of the motivation of this study, we are interested in the relation between the local time of a random walk and a Gaussian free field. 
As a well known result,  the generalized second Ray-Knight theorem makes the local time of  random walk and a Gaussian free field closely linked as follows. 
To state this theorem, we consider a reversible continuous-time random walk $\{S_t\}_{t\ge0 }$ on a graph $G$, 
the total conductance (or weight) $\{\lambda_x\}_{x\in G}$ (See the definition in \cite{Sz}),   
the local time 
$$\tilde{K}(t,x):=\lambda_x^{-1}\int_0^t 1_{\{S_l=x \}} ds,$$
 and 
 $$\tilde{\tau}_t:=\inf\{s: \tilde{K}(s,0)>t\}.$$
Pick a certain point in $G$ as the origin. 
 Let $\{\phi(x)\}_{x \in G}$ be the Gaussian free field with mean zero and 
Cov$(\phi(x),\phi(y))=E^x[\tilde{K}(T_0,y)]$, 
where $T_0$ is stopping time to the origin. 
\begin{thm}[\cite{Eisen} : 
The generalized second Ray-Knight theorem]
\begin{align*}
&\text{The law of }\{\tilde{K}(\tilde{\tau}_t,x)+\frac{1}{2}\phi(x)^2 \}_{x\in G} \text{ under } P \times \mathbf{P}\\
&\text{is the same as that of }\{ \frac{1}{2}(\phi(x )+ \sqrt{2t})^2 \}_{x\in G} \text{ under }\mathbf{P},
\end{align*}
where $\mathbf{P}$ is the probability of the Gaussian free field and 
$P$ is that of the random walk starting the origin. 
\end{thm}
A simple sorting of above yields the following.
\begin{crl}[\cite{Eisen}]
\begin{align*}
\{\frac{\tilde{K}(\tilde{\tau}_t,x)-t}{\sqrt{2t}} \}_{x\in G} 
\to \{\phi(x )\}_{x\in G} \text{ in law as } t\to \infty.
\end{align*}
\end{crl}
We consider that this corresponds to the central limit theorem for the local time $\tilde{K}(\tilde{\tau}_t,x)$, and hence, 
the limit of $\tilde{K}(\tilde{\tau}_t,x)$ as  the central limit theorem is the corresponding Gaussian free field. 
Then, we have the question: 
how far are the distribution of the local time from that of the Gaussian free field at each time? 
Then, we are interested in two processes and
 especially study the relation between special points of a discrete-time simple random walk in 
$\mathbb{Z}^2$ or $\mathbb{Z}^2_n$(:=$\mathbb{Z}^2/n\mathbb{Z}^2 $) 
where the local time is large or small 
and special points  in $\mathbb{Z}^2_n$ where the corresponding Gaussian free field takes large value. 
To explain properties of such points of a simple random walk in $\mathbb{Z}^d$, we give some definitions of a simple  random walk. 
(In the sequel, we consider only discrete-time simple random walk in $\mathbb{Z}^d$ (or $\mathbb{Z}^2_n$). ) 
Let $\{S_k\}_{k=0}^{\infty}$ be a simple random walk on the $d$-dimensional square lattice. 
Let $d(x,y)$ be the Euclidean distance for $x,y \in  \mathbb{Z}^d$. 
Let $D(x,r):=\{y\in \mathbb{Z}^d: d(x,y)\le r\}$ and for $G\subset \mathbb{Z}^d$, 
$\partial G:=\{y\in G: d(x,y)=1\text{ for some } x\in G^c \}$. 
Let $P^x$ denote the probability measure for the simple random walk starting at $x$. 
We simply write $P$ for $P^0$.  
Let $K(n,x)$ be the number of visits of the simple random walk to $x$ until time $n$, 
that is, $K(n,x)=\sum_{i=0}^n1_{\{S_i=x\}}$. 
For $D\subset \mathbb{Z}^d$, let 
$T_D:=\inf \{m\ge1: S_m\in D\}$. 
In particular, we write $T_{x_1,...,x_j}$ for $T_{\{x_1,...,x_j\}}$. 
Let $\tau_n:=\inf \{m\ge 0: S_m\in \partial D(0,n)\}$. 
In addition, $\lceil a \rceil$ denotes the smallest integer $n$ with $n \ge a$.

\section{Some estimates for favorite points in $\mathbb{Z}^d$ with $d\ge 2$}
We call the most frequently visited site among all of the random walk range (up to a specific time) favorite point. 
About fifty years ago, Erd\H{o}s and Taylor \cite{er} proposed a problem concerning the simple random walk in $\mathbb{Z}^d$: 
how many times does the random walk revisit the favorite point? 
In fact, \cite{er} showed
that  for the simple random walk in $d\ge3$
\begin{align*}
\lim_{n \to \infty} \frac{\max_{x\in {\mathbb Z}^d}K(n,x)}{\log n}= \frac{1}{-\log P(T_0<\infty)}\quad \text{ a.s.}
\end{align*}
In addition, for $d=2$, they obtained 
\begin{align*}
\frac{1}{4\pi} \le 
\liminf_{n \to \infty} \frac{\max_{x\in {\mathbb Z}^2}K(n,x)}{(\log n)^2}\le
\limsup_{n \to \infty} \frac{\max_{x\in {\mathbb Z}^2}K(n,x)}{(\log n)^2}\le \frac{1}{\pi} \quad \text{ a.s.},
\end{align*}
and conjectured that the limit exists and equals $1/\pi$ a.s.  
Forty years later,  Dembo et al. \cite{Dembo} verified this conjecture, 
that is, they showed that for a simple random walk in ${\mathbb{Z}^2}$,
\begin{align*}
\lim_{n \to \infty} \frac{\max_{x\in {\mathbb Z}^2}K(\tau_n,x)}{(\log n)^2}
=\lim_{n \to \infty} \frac{4\max_{x\in {\mathbb Z}^2}K(n,x)}{(\log n)^2}
= \frac{4}{\pi} \quad a.s.
\end{align*}
In addition, for $0<\alpha<1$, \cite{Dembo, Dembo2} defined the set of  $\alpha$-favorite points 
in $\mathbb{Z}^2$ such that
\begin{align*}
 \Psi_n(\alpha):=\bigg\{x: K(\tau_n,x)\ge \bigg\lceil \frac{4\alpha}{\pi} (\log n)^2 \bigg\rceil \bigg\}. 
\end{align*}
It is known that the law of large numbers for the random walk range in $\mathbb{Z}^2$ holds by \cite{jain}. 
\cite{Dembo} showed the  law of large numbers for $ \Psi_n(\alpha)$. 
After this, \cite{rosen} made another proof of results of \cite{Dembo}. 
In addition, there are many works studying properties of the favorite points for any dimension such as the number of visit to it and the location of it. 
For example, in \cite{okada1} we showed that the favorite points of the simple random walk in $\mathbb{Z}^d$ with $d\ge2$ does not appear in the boundary of the random walk range from some time on a.s. 
Additional open problems concerning geometric structure of the favorite points are raised by Erd\H{o}s and R\'ev\'esz \cite{er2, er3} and Shi and T\'oth \cite{shi} but almost no definite solution to them is known for multi-dimensional walks. 

Now, we introduce results of \cite{okada1}. 
Set $R(n) := \{S_0,S_1, \ldots, S_n\}$ as the random walk range until time $n$. 
Moreover, we set $$M(n):=\max_{x \in \partial R(n)}K(n,x).$$
The first theorem provides us with sharp asymptotic behavior of $M(n)$.
\begin{thm}[\cite{okada1}]\label{k1}
For $d\ge2$
\begin{align*}
\lim_{n \to \infty} \frac{M(n)}{\log n}=\beta_d \quad \text{ a.s.},
\end{align*}
where
\begin{align*}
\beta_d=\frac{1}{- \log P(T_0<T_b)}
\end{align*}
and $b$ is a neighbor of the origin. 
\end{thm}
To compare to $|\Psi_n(\alpha)|$, we  define $\Theta_n (\delta)$ for $n \in \mathbb{N}$ and $0<\delta<1$ as
\begin{align*}
\Theta_n(\delta):=| \{ x\in \partial R(n) :
\frac{K(n,x)}{\log n}\ge \beta_d\delta \}|.
\end{align*}
In fact, it is known that the law of large numbers for the boundary of the range of a transient random walk holds by \cite{okada}. 
Then, we have a question : does the law of large numbers for $\Theta_n(\delta)$ hold? 
The following corresponds to the answer. 
\begin{thm}[\cite{okada1}]\label{k2}
For $d\ge2$ and $0<\delta<1$,
\begin{align*}
\lim_{n \to \infty} \frac{\log \Theta_n(\delta)}{\log n}=1-\delta \quad \text{ a.s.}
\end{align*}
\end{thm}
Next, we will extend these results to that of some general boundary. 
Fix $M\in \mathbb{N}$. 
Let ${\cal H}={\cal H}(M):=\{H\subset \mathbb{Z}^d: \{0\} \notin H, |H|\le M, P(T_x < T_H )>0 \quad \forall x \in H^c \}$. 
For any $H\subset \mathbb{Z}^d$ and $G\subset \mathbb{Z}^d$, let
$$\partial_H G:= \{y\in G: y+H \subset G^c \} .$$
Note that if we let  ${\cal N}(0)$ be the set of  neighbors of the origin, $\cup_{b \in {\cal N}(0)} \partial_{\{b\}} R(n)=\partial R(n)$ holds. 
In addition, we need the condition ``$P(T_x < T_H )>0$ $\forall x \in H^c$" in the proof of Lemma \ref{hh+} below. 
This condition verifies that $H$ does't not surround the origin. 
\begin{thm}\label{p0}
For $d\ge3$ and $H\in {\cal H}$
\begin{align*}
\lim_{n \to \infty} \frac{|\partial_H R(n)|}{n}=q \quad \text{ a.s.},
\end{align*}
where
\begin{align*}
q:=P((\{S'_m\}_{m=0}^\infty \cup \{S_m\}_{m=0}^\infty) \cap H=\emptyset \text{ and }0 \in \{S_m\}_{m=1}^\infty)
\end{align*}
and 
$\{S'_m\}_{m=0}^\infty$ denotes an independent copy of $\{-S_m\}_{m=0}^\infty$. 
\end{thm}
Compare this to Theorem $2.1$ in \cite{okada}. 
It is trivial that almost the same argument as in the proof of Theorem $2.1$ in \cite{okada}  yields this result. 
\begin{rem}
Note that any choice of $H\in {\cal H}$ does not realize $\partial _H G= \partial G$ in general. 
However, Theorem \ref{p0} yields  Theorem $2.1$ in \cite{okada} easily. 
Note that for $H_1$, $H_2\subset \mathbb{Z}^d$ and $G \subset \mathbb{Z}^d$
\begin{align*}
\partial_{H_1} G \cap \partial_{H_2} G 
= \partial_{H_1 \cup H_2} G.
\end{align*}
Set ${\cal U}_k := \{  A \subset {\cal N} (0) :  |A |= k \}$. 
Then, the inclusion-exclusion identity yields
\begin{align*}
\sum_{k=1}^{2d} 
\sum_{H \in {\cal U}_k }
(-1)^{k+1}
|\partial_H R (n) | = | \partial R (n) |.
\end{align*}
In addition, 
 \begin{align*}
&\sum_{k=1}^{2d} 
\sum_{H \in {\cal U}_k }
(-1)^{k+1}
P((\{S'_m\}_{m=0}^\infty \cup \{S_m\}_{m=0}^\infty) \cap H_k=\emptyset \text{ and }0 \in \{S_m\}_{m=1}^\infty)\\
=&P((\{S'_m\}_{m=0}^\infty \cup \{S_m\}_{m=0}^\infty) \not\supset {\cal N}(0) \text{ and }0 \in \{S_m\}_{m=1}^\infty). 
 \end{align*}
 Therefore, the desired result holds. 
\end{rem}
For $H\in {\cal H}$, we set $$M_H(n):=\max_{x \in \partial_H R(n)}K(n,x).$$
The following corresponds to the extension of Theorem \ref{k1}. 
\begin{thm}\label{p1}
For $d\ge3$ and $H\in {\cal H}$,
\begin{align*}
\lim_{n \to \infty} \frac{M_H(n)}{\log n}=\beta_d(H) \quad \text{ a.s.},
\end{align*}
where
\begin{align*}
\beta_d(H)=\frac{1}{- \log P(T_0<T_H)}. 
\end{align*}
\end{thm}
To compare to $\Theta_n (\delta)$, we define $\Theta_n (\delta,H)$ for $n \in \mathbb{N}$ and $0<\delta<1$ as
\begin{align*}
\Theta_n(\delta,H):=|\{ x\in \partial_H R(n) :
\frac{K(n,x)}{\log n}\ge \beta_d(H)\delta \}|.
\end{align*}
In view of Theorem \ref{p1} together with Theorems \ref{k1} and \ref{k2}, we have a question : 
does the law of large numbers for $\Theta_n(\delta,H)$ hold?  
The following corresponds to the answer. 
\begin{thm}\label{p2}
For $d\ge3$, $H\in {\cal H}$ and $0<\delta<1$,
\begin{align*}
\lim_{n \to \infty} \frac{\log \Theta_n(\delta,H)}{\log n}=1-\delta \quad \text{ a.s.}
\end{align*}
\end{thm}
\begin{rem}
We extend Theorems \ref{p1} and \ref{p2} to Theorems \ref{k1} and \ref{k2} 
by using the same argument as in Remark after Theorem \ref{p0}. 
\end{rem}
As we will see, the proof of Theorem \ref{p1} is a minor modification of that of Theorem \ref{k1}. 
Note that Theorem \ref{k2} follows from assertions developed in the proof of Theorem \ref{k1}. 
Analogously, Theorem \ref{p2} follows in the same manner once we proved Theorem \ref{p1}. 
Thus we argue only the proof of Theorem \ref{p1}. 

To show Theorem \ref{p1}, let $T_x^0=\inf\{j\ge0: S_j=x\}$ and for $p\ge1$,
 \begin{align*}
 T_x^p=\inf\{ j>T_x^{p-1}: S_j=x\}
 \end{align*}
with the convention  $\inf \emptyset =\infty$. 
To show the upper bound, we provide the following:
\begin{lmm}\label{up}
For $\beta>0$ and $ H\in {\cal H}$ there exists $C>0$ such that for any $n\in \mathbb{N}$
\begin{align*}
E[\tilde{\Theta}_n(\beta,H)]\le Cn^{1-\frac{\beta}{\beta_d(H)}},
\end{align*}
where 
\begin{align*}
\tilde{\Theta}_n(\beta,H)
:=|\{ x\in\partial_H R(T_x^{\lceil \beta\log  (n/2) \rceil }) :K(n,x)\ge \lceil \beta \log \frac{n}{2}\rceil \}|.
\end{align*}
\end{lmm}
\begin{rem}
Note that Lemma \ref{up}  corresponds to Lemma $3.1$ in \cite{okada1}. 
Therefore, if we obtain Lemma \ref{up}, by the same argument as in the proof of Proposition $3.1$ in \cite{okada1}, 
 the desired upper bound follows with minor modifications. 
\end{rem}
\begin{proof}
First, we introduce the elementary property corresponding to $(3)$ in \cite{okada1}.  
For any intervals $I_0$, $I_1$, $I_2\subset {\mathbb N}\cup \{0\}$ with $I_0 \subset I_1 \subset I_2$, it holds that 
\begin{align*}
R(I_0)\cap \partial_H R (I_2)  \subset \partial_H R (I_1).
\end{align*}
Then, if we change $\partial$ in the proof of Lemma $3.1$ in \cite{okada1} to  $\partial_H$, 
we obtain the result with minor modifications. 
\end{proof}

Next, to show the lower bound, we will define the following. 
For $k\in {\mathbb{N}}$ and $\beta<\beta_d(H)$, let 
$h_k=\beta\log P(T_0<T_H\wedge k)+1$.  
Let $u_n:=\lceil \exp(n^2)\rceil$ and 
$$I_n:=[\frac{u_{n-1}}{n^2}, u_{n-1}-k \lceil \beta_d(H)  n^2\rceil]\cap {\mathbb{N}}.$$ 
For any $l\in I_n$, we introduce the event  $A_{l,n}$ defined by 
\begin{align*}
E_{l,n}:=\{ T_{S_l}^{j}-T_{S_l}^{j-1}<k \text{ for any }1\le j\le \lceil \beta n^2\rceil \},
\end{align*}
and
\begin{align*}
A_{l,n}:=\{S_l \in R(l-1)^c \cap \partial_H R(u_{n}) \}\cap E_{l,n}.
\end{align*}
Then, we set
\begin{align*}
Q_n:=\sum_{l \in I_n }1_{A_{l,n}}. 
\end{align*}
\begin{lmm} \label{hh+}
Let $\beta<\beta_d$ and take $k\in{\mathbb{N}}$.  
Then, there exists $c>0$ such that for any $n\in{\mathbb{N}}$, 
\begin{align*}
EQ_n \ge c\exp(h_kn^2-2n).
\end{align*}
\end{lmm}
\begin{lmm} \label{hh}
Let $\beta<\beta_d$ and take $k\in{\mathbb{N}}$.  
Then, there exists $C>0$ such that for any $n\in{\mathbb{N}}$, 
\begin{align*}
\mathrm{Var} (Q_n) \le C\exp(2h_kn^2-4n)\times \frac{1}{n^{10}}.
\end{align*}
\end{lmm}
\begin{rem}
Note that Lemmas \ref{hh+} and \ref{hh}  correspond to Lemmas $4.2$ and $4.3$ in \cite{okada1}. 
Therefore, if we change $\partial_b$  in the proof of Proposition $4.1$ in \cite{okada1} to  $\partial_H$, 
Lemmas \ref{hh+} and \ref{hh} yield the lower bound with minor modifications. 
\end{rem}
\begin{proof}[Proof of Lemmas \ref{hh+} and \ref{hh}]
Again, note that Lemmas \ref{hh+} and \ref{hh}  correspond to Lemmas $4.2$ and $4.3$ in \cite{okada1}. 
Especially, we need the condition ``$P(T_x < T_H )>0$ $\forall x \in H^c$" in ${\cal H}$ in the proof of Lemma \ref{hh+}. 
This condition yields $P(T_0 \wedge T_H=\infty),P( T_H=\infty)>0$.  
Lemma $4.4$, Corollary $4.1$, Lemmas $4.5$ and $4.6$ and the simple summation over $I_n$ in \cite{okada1} yield Lemmas $4.2$ and $4.3$ in \cite{okada1}. 
If we change $b$  in the proof  of \cite{okada1} to  $H$, 
we obtain results corresponding to Lemma $4.4$, Corollary $4.1$, Lemmas $4.5$ and $4.6$ in \cite{okada1} with minor modifications. 
Especially, the facts $P(T_0 \wedge T_H=\infty),P( T_H=\infty)>0$ are used to show the assertion corresponding to Lemma $4.5$. 
Other arguments are same as that of \cite{okada1}. 
Therefore, desired results follow. 
 \end{proof}
 \begin{proof}[Proof of Theorem \ref{p1}]
 If we change $b$  in the proof  of \cite{okada1} to  $H$, 
we obtain results corresponding to Proposition $4.1$ in \cite{okada1} with minor modifications. 
\end{proof}
Now, we provide two open problems for $d=2$. 
\begin{open}
For $d=2$, is it true that results corresponding to Theorems \ref{p1} and \ref{p2} hold?
\end{open}
We believe that it is true and explain why it is difficult to solve it. 
If we solve the lower bound of this by the same argument as \cite{okada1}, 
we need to estimate the following: 
 for any $n\in \mathbb{N}$, 
 \begin{align*}
P(T_H > \lceil \frac{u_n}{n} \rceil),\quad
P(T_0\wedge T_H > \lceil \frac{u_n}{n} \rceil)
\end{align*}
and  for any $n\in \mathbb{N}$ and $x\in {\mathbb{Z}^{2}}$  
with $0<|x|< n\sqrt{u_{n-1}}$, 
 \begin{align*}
P(T_H \wedge T_{x+H}> \lceil \frac{u_n}{n} \rceil).
\end{align*}
Note that these correspond to (52), (53) or Lemma $4.7$ in \cite{okada1}. 
In Corollary 2 and (1.2)  of \cite{kesten}, they showed that for $H\subset \mathbb{Z}^2$,
\begin{align*}
\sum_{i\in H} P^i (T_H\ge n) \sim \frac{\pi}{\log n}, 
\end{align*}
where $a_n \sim c_n$ means $a_n/c_n \to 1$ $(n\to \infty)$ for sequences $a_n$ and $c_n$. 
In addition, \cite{okada} yields 
\begin{align*}
P^i (T_0 \wedge T_b \ge n) \sim \frac{\pi}{2\log n}, 
\end{align*}
where $b$ is a neighbor of the origin. 
However, it is difficult to estimate the individual term for a general $H\subset \mathbb{Z}^2$, 
that is, $P^i (T_H\ge n)$ for $i\in H$. 

In addition, we will consider a certain extension of Theorem \ref{k1} for $d=2$. 
\begin{open}
Consider $\tilde{D}_n\in  {\cal H}$  with $\tilde{D}_n\subset D(0, a_n)^c$ and $a_n \to \infty$ as $n \to \infty$. 
Then, how is the asymptotic of $M_{\tilde{D}_n}(n)$ with probability one? 
\end{open}
We are especially interested in the case 
that $\tilde{D}_n\in  {\cal H}$ with $\tilde{D}_n\subset D(0, a_n)^c \cap D(0, 2a_n)$ and  $a_n=n^{\beta}$ for $0<\beta<1/2$. 
In fact, \cite{Dembo1} showed the largest disc completely covered until $\tau_n$ by the simple random walk in $\mathbb{Z}^2$ is $n^{1/2+o(1)}$. 
Then, we conjecture that this argument yields that if we set $a_n\ge n^{1/2+\epsilon}$ for some $\epsilon>0$, it holds that 
\begin{align*}
M_{\tilde{D}_n}(n)=\max_{x\in \mathbb{Z}^2} K(\tau_n,x) \quad 
\text{for all sufficiently large }n\in \mathbb{N}\quad \text{ a.s.}
\end{align*}

\section{Some estimates for favorite points and late points of simple random walks and high points of Gaussian free fields  in two dimensions}
First, we explain the high points of the Gaussian free field in $\mathbb{Z}^2_n$. 
Originally, 
Bolthausen, Deuschel and Giacomin \cite{Bol} showed that in probability
\begin{align*}
\lim_{n \to \infty}\frac{ \max_{x \in \mathbb{Z}_n^2 }\phi_n(x)}{\log n}=2\sqrt{\frac{2}{\pi}},
\end{align*}
where $\{\phi_n(x)\}_{x \in \mathbb{Z}_n^2}$ is the Gaussian free field defined in \cite{Bol,ol}. 
In what follows, for $0<\alpha<1$, we define the set of $\alpha$-high points of the Gaussian free field by
\begin{align*}
{\cal V}_n(\alpha):=\bigg\{ x\in {\mathbb{Z}^2_n} :
\frac{\phi_n(x)^2}{2} \ge \frac{4\alpha}{\pi} (\log n)^2 \bigg\}.
\end{align*}

Next, we explain late points of a simple random walk in $\mathbb{Z}^2_n$. 
Originally, 
Dembo, Peres, Rosen and Zeitouni \cite{Dembo3} showed that
for a simple random walk in ${\mathbb{Z}^2_n}$ in probability
\begin{align*}
\lim_{n \to \infty} \frac{\max_{x\in {\mathbb Z}^2_n}T_x }{(n\log n)^2}= \frac{4}{\pi}.
\end{align*}
(Now, we consider $T_x$ as the stopping time to $x$ of a simple random walk in $\mathbb{Z}^2_n$.) 
In what follows, for $0<\alpha<1$, we define the set of $\alpha$-late points in $\mathbb{Z}^2_n$ 
such that 
\begin{align*}
{\cal L}_n(\alpha):=\bigg\{ x\in {\mathbb{Z}^2_n} :
\frac{T_x}{(n\log n)^2}\ge \frac{4\alpha}{\pi} \bigg\}.
\end{align*}

For any $0< \alpha,\beta <1$ and $\{G_n\}_{n=1}^{\infty}$ with $G_n\subset \mathbb{Z}^2$ (or $\mathbb{Z}_n^2$), let 
\begin{align*}
&Q_j=Q_j (G_n):=\lim_{n\to \infty}\frac{\log 
|\{ (x_1,...,x_j)\in G_n^j:d(x_i,x_l)\le n^{\beta} \text{ for any }1\le i,l \le j \}|}{\log n}\\
&\hat{Q}_j=\hat{Q}_j (G_n):=\lim_{n\to \infty}\frac{\log 
E[|\{ (x_1,...,x_j)\in G_n^j:d(x_i,x_l)\le n^{\beta} \text{ for any }1\le i,l \le j \}||]}{\log n}
\end{align*}
if the right hand sides exist. 
For any $0< \alpha,\beta <1$, set
\begin{align*}
&\rho_2(\alpha, \beta):=
\begin{cases}
 2+2\beta-\frac{4\alpha}{2-\beta}&(\beta\le 2(1-\sqrt{\alpha}))
\\
8(1-\sqrt{\alpha})-4(1-\sqrt{\alpha})^2/\beta&(\beta\ge 2(1-\sqrt{\alpha})),
\end{cases}
\\
&\hat{\rho}_2(\alpha, \beta):=
\begin{cases}
2+2\beta-\frac{4\alpha}{2-\beta}&(\beta\le 2-\sqrt{2\alpha})
\\
 6-4\sqrt{2\alpha}&(\beta\ge 2-\sqrt{2\alpha}).
\end{cases}
\end{align*}
All the known results for $Q_2({\cal L}_n(\alpha))$, $Q_2({\cal V}_n(\alpha))$ and $Q_2(\Psi_n(\alpha))$ (or the corresponding ones for $\hat{Q}_2$) are summarized in Table $1$ below.
\begin{table}[h]
\caption{Known results for $Q_j$ and $\hat{Q}_j$}
  \begin{tabular}{|l||c|c|c|c|} \hline
     &  $Q_2$ in prob. & $\hat{Q}_2$ & $Q_j$ in prob.  $\forall  j\in \mathbb{N}$ & $\hat{Q}_j$ $\forall  j\in \mathbb{N}$ \\ \hline \hline
    ${\cal L}_n(\alpha)$  &$\rho_2(\alpha, \beta)$ : \cite{Dembo2} &$\hat{\rho}_2(\alpha, \beta)$ : \cite{ho}  & Open prob. by \cite{Dembo*, Dembo**} & Unsolved \\ \hline 
    ${\cal V}_n(\alpha)$  & $\rho_2(\alpha, \beta)$ : \cite{ol} &$\hat{\rho}_2(\alpha, \beta)$ : \cite{ol} &Unsolved &Unsolved \\ \hline  
    ${\Psi}_n(\alpha)$  & Open prob. by \cite{Dembo2}& Unsolved &Unsolved &Unsolved \\ \hline
  \end{tabular}
\end{table}

If we write $"A: [b]"$ in the above table, the value of $Q_2$ (or $\hat{Q}_2$) is identified with 
$A$ and $[b]$ is the reference which solved this problem. 
In \cite{okada2}, we solve the problems concerning $Q_2(\Psi_n(\alpha))$ or $\hat{Q}_2(\Psi_n(\alpha))$. 
Our result shows that all three exponents in almost sure sense (that is, $Q_2$) coincide with one another. 
We also estimate this value in average (that is, $\hat{Q}_2$) and obtain a similar coincidence. 
In addition, \cite{okada3} solves the problems concerning $Q_j({\cal L}_n(\alpha))$ or $\hat{Q}_j({\cal L}_n(\alpha))$ for $j \in \mathbb{N}$. 
As stated in Theorem $1.1$,  \cite{Eisen} gave a powerful equivalence in law.  
Ding, Lee, and Peres \cite{Ding1, Ding2} gave a strong connection between the expected maximum of the Gaussian free field and the expected cover time. 
Then, we are interested in the stronger relation between $\Psi_n(\alpha)$ and $ {\cal V}_n(\alpha)$ (or  ${\cal L}_n(\alpha)$) and 
search for the difference and the similarity. 
We suggest the following problem. 
\begin{open}
Consider the slightly different Gaussian free field from that defined by \cite{Bol, ol}, 
which is defined on $D(0,n)$ and whose covariance is $E^x[K(\tau_n,y)]$ for $x,y\in D(0,n)$. 
In addition, consider the corresponding ${\cal V}_n(\alpha)$. 
Is there the sequence of the coupling of the Gaussian free field in $D(0,n)$ and simple random walk in $\mathbb{Z}^2$ such that 
for any $\epsilon>0$, $\epsilon<\alpha<1$ and all sufficiently large $n\in \mathbb{N}$, 
\begin{align}\label{g*}
&P(\Psi_n(\alpha)\subset {\cal V}_n(\alpha-\epsilon) )=1\\
\label{g**}
\text{or }\quad&P( {\cal V}_n(\alpha) \subset\Psi_n(\alpha-\epsilon) )=1?
\end{align}
\end{open}


\section*{Acknowledgments. }
We are grateful to Prof. Kazumasa Kuwada for interesting discussion. 
In addition,  we thank referees for careful reading our manuscript and for giving useful comments.

\end{document}